\title{Regularity partitions and the topology of graphons}
\author{L\'aszl\'o Lov\'asz\footnote{Research supported by OTKA grant No.~67867 and ERC Advanced Research Grant No.~227701}\\
Institute of Mathematics\\
E\"otvos Lor\'and University\\
and\\
Bal\'azs Szegedy\\
Department of Mathematics\\
University of Toronto}

\documentclass{article}
\usepackage{amssymb,amsmath,hyperref,theorem}
\nonstopmode

\newtheorem{theorem}{Theorem}[section]
\newtheorem{prop}[theorem]{Proposition}
\newtheorem{lemma}[theorem]{Lemma}
\newtheorem{claim}[theorem]{Claim}
\newtheorem{corollary}[theorem]{Corollary}

\theorembodyfont{\rmfamily}
\newtheorem{example}{Example}
\newtheorem{conj}[theorem]{Conjecture}

\newtheorem{remark}[theorem]{Remark}

\newenvironment{proof}{\medskip\noindent{\bf Proof. }}{\hfill$\square$\medskip}
\newenvironment{proof*}[1]{\medskip\noindent{\bf Proof of #1.}}{\hfill$\square$\medskip}

\long\def\killtext#1{}

\begin{document}
\addtolength{\baselineskip}{3pt} \setlength{\oddsidemargin}{0.2in}

\def\hom{{\rm hom}}
\def\iso{{\rm iso}}
\def\inj{{\rm inj}}
\def\surj{{\rm sur}}
\def\Inj{{\rm Inj}}
\def\neighb{{\rm neighb}}
\def\inj{{\rm inj}}
\def\ind{{\rm ind}}
\def\sur{{\rm sur}}
\def\PAG{{\rm PAG}}
\def\mul{{\rm mult}}
\def\flat{{\rm flat}}
\def\supp{{\rm supp}}
\def\tbp{t^{\sf b}}

\def\eps{\varepsilon}
\def\thet{\vartheta}

\def\CUT{\text{\rm CUT}}
\def\IO{{\infty\to1}}
\def\GR{\text{\rm GR}}
\def\CLQ{\text{\rm CLQ}}
\def\LT{{\text{\rm LEFT}}}
\def\RT{{\text{\rm RIGHT}}}
\def\Haus{{\rm Hf}}

\def\Fi{\mathbf{\Phi}}

\def\maxcut{{\sf maxcut}}
\def\MaxCut{{\sf MaxCut}}
\def\id{{\rm id}}
\def\irreg{{\rm irreg}}

\def\tv{{\rm tv}}
\def\tr{{\rm tr}}
\def\cost{\hbox{\rm cost}}
\def\val{\hbox{\rm val}}
\def\rk{{\rm rk}}
\def\diam{{\rm diam}}
\def\Ker{{\rm Ker}}
\def\simi{{\rm sim}}

\def\eul{\text{\sf eul}}
\def\Eu{\text{\sf euln}}
\def\Tu{\text{\sf tut}}
\def\Cr{\text{\sf chr}}
\def\Fl{\text{\sf flo}}
\def\Fs{\text{\sf f}}
\def\Ham{\text{\sf ham}}
\def\Pmg{\text{\sf pmg}}
\def\Ecp{\text{\sf ecp}}
\def\Mch{\text{\sf match}}
\def\Edge{\text{\sf edge}}
\def\Subg{\text{\sf subg}}
\def\Loop{\text{\sf loop}}
\def\Sflow{\text{\sf sflo}}

\def\Pr{{\sf P}}
\def\E{{\sf E}}
\def\Var{{\sf Var}}
\def\Ent{{\sf Ent}}
\def\PD{{\sf Pd}}
\def\SYM{{\sf Sym}}

\def\T{{^\top}}
\def\wt{\widetilde}
\def\wh{\widehat}

\def\AA{\mathcal{A}}\def\BB{\mathcal{B}}\def\CC{\mathcal{C}}
\def\DD{\mathcal{D}}\def\EE{\mathcal{E}}\def\FF{\mathcal{F}}
\def\GG{\mathcal{G}}\def\HH{\mathcal{H}}\def\II{\mathcal{I}}
\def\JJ{\mathcal{J}}\def\KK{\mathcal{K}}\def\LL{\mathcal{L}}
\def\MM{\mathcal{M}}\def\NN{\mathcal{N}}\def\OO{\mathcal{O}}
\def\PP{\mathcal{P}}\def\QQ{\mathcal{Q}}\def\RR{\mathcal{R}}
\def\SS{\mathcal{S}}\def\TT{\mathcal{T}}\def\UU{\mathcal{U}}
\def\VV{\mathcal{V}}\def\WW{\mathcal{W}}\def\XX{\mathcal{X}}
\def\YY{\mathcal{Y}}\def\ZZ{\mathcal{Z}}

\def\Ab{\mathbf{A}}\def\Bb{\mathbf{B}}\def\Cb{\mathbf{C}}
\def\Db{\mathbf{D}}\def\Eb{\mathbf{E}}\def\Fb{\mathbf{F}}
\def\Gb{\mathbf{G}}\def\Hb{\mathbf{H}}\def\Ib{\mathbf{I}}
\def\Jb{\mathbf{J}}\def\Kb{\mathbf{K}}\def\Lb{\mathbf{L}}
\def\Mb{\mathbf{M}}\def\Nb{\mathbf{N}}\def\Ob{\mathbf{O}}
\def\Pb{\mathbf{P}}\def\Qb{\mathbf{Q}}\def\Rb{\mathbf{R}}
\def\Sb{\mathbf{S}}\def\Tb{\mathbf{T}}\def\Ub{\mathbf{U}}
\def\Vb{\mathbf{V}}\def\Wb{\mathbf{W}}\def\Xb{\mathbf{X}}
\def\Yb{\mathbf{Y}}\def\Zb{\mathbf{Z}}

\def\ab{\mathbf{a}}\def\bb{\mathbf{b}}\def\cb{\mathbf{c}}
\def\db{\mathbf{d}}\def\eb{\mathbf{e}}\def\fb{\mathbf{f}}
\def\gb{\mathbf{g}}\def\hb{\mathbf{h}}\def\ib{\mathbf{i}}
\def\jb{\mathbf{j}}\def\kb{\mathbf{k}}\def\lb{\mathbf{l}}
\def\mb{\mathbf{m}}\def\nb{\mathbf{n}}\def\ob{\mathbf{o}}
\def\pb{\mathbf{p}}\def\qb{\mathbf{q}}\def\rb{\mathbf{r}}
\def\sb{\mathbf{s}}\def\tb{\mathbf{t}}\def\ub{\mathbf{u}}
\def\vb{\mathbf{v}}\def\wb{\mathbf{w}}\def\xb{\mathbf{x}}
\def\yb{\mathbf{y}}\def\zb{\mathbf{z}}

\def\Abb{\mathbb{A}}\def\Bbb{\mathbb{B}}\def\Cbb{\mathbb{C}}
\def\Dbb{\mathbb{D}}\def\Ebb{\mathbb{E}}\def\Fbb{\mathbb{F}}
\def\Gbb{\mathbb{G}}\def\Hbb{\mathbb{H}}\def\Ibb{\mathbb{I}}
\def\Jbb{\mathbb{J}}\def\Kbb{\mathbb{K}}\def\Lbb{\mathbb{L}}
\def\Mbb{\mathbb{M}}\def\Nbb{\mathbb{N}}\def\Obb{\mathbb{O}}
\def\Pbb{\mathbb{P}}\def\Qbb{\mathbb{Q}}\def\Rbb{\mathbb{R}}
\def\Sbb{\mathbb{S}}\def\Tbb{\mathbb{T}}\def\Ubb{\mathbb{U}}
\def\Vbb{\mathbb{V}}\def\Wbb{\mathbb{W}}\def\Xbb{\mathbb{X}}
\def\Ybb{\mathbb{Y}}\def\Zbb{\mathbb{Z}}

\def\R{{\mathbb R}}
\def\Q{{\mathbb Q}}
\def\Z{{\mathbb Z}}
\def\N{{\mathbb N}}
\def\C{{\mathbb C}}
\def\U{{\mathbb U}}
\def\Ge{{\mathbb G}}
\def\Ha{{\mathbb H}}

\def\one{{\sf\bf 1}}

\maketitle

\tableofcontents

\begin{abstract}
We highlight a topological aspect of the graph limit theory. Graphons
are limit objects for convergent sequences of dense graphs. We
introduce the representation of a graphon on a unique metric space
and we relate the dimension of this metric space to the size of
regularity partitions. We prove that if a graphon has an excluded
induced sub-bigraph then the underlying metric space is compact and
has finite packing dimension. It implies in particular that such
graphons have regularity partitions of polynomial size.
\end{abstract}

\section{Introduction}

One can define convergence of a growing graph sequence
\cite{BCLSSV,BCLSV0,BCLSV1}, and construct a limit object to such a
sequence \cite{LSz1} in the form of a symmetric measurable function
$W:~J\times J\to[0,1]$, where $J$ is any probability space (one may
assume here that $J=[0,1]$ with the Lebesgue measure, but this is not
always convenient). We call the pair $(J,W)$ a {\it graphon}.

The goal of this paper is to show that one can introduce also a
topology on $J$ (in fact, a metric), and that topological properties
of this space are related to combinatorial properties of the graphon
(or of the graphs whose limit it represents). A related metric was
introduced in \cite{LSz3}, and the topology on $J$ was used in
\cite{LSz6}.

The theory of graph limits is tied to the Regularity Lemma of
Szemer\'edi \cite{Szem1,Szem2} in several ways. In \cite{LSz3} it was
shown that the Regularity Lemma is equivalent to the compactness of
the space of graphons in an appropriate metric, and also to a
``dimensionality'' of particular graphons. This paper relates to the
latter result.

The metric in question is simply the $L_1$ metric on functions
$W(x,.)$, $x\in J$. This metric itself can be weird (it may not even
be defined on all points of $J$). We show in Section \ref{TOPGRAPH}
that that every graphon is ``equivalent'' (technically: weakly
isomorphic, see the end of Section \ref{PRELIM}) to a graphon $(J,W)$
with special properties: $J$ is a complete separable metric space,
and the probability measure on $J$ has full support. We call such
graphons {\it pure}. We also prove that the pure version of a graphon
is uniquely determined up to changing the function $W$ on a $0$-set
in each row. We define another metric in which $J$ is compact, and
characterize the cases when the two define the same topology. We
prove that several important functions defined on $J$ are continuous
in this topology, which shows that it is indeed the ``right''
topology to define on $J$.

In Section \ref{THIN} we show that topological properties of pure
graphons are related to their graph-theoretic properties. Our main
result states tha {\it if we exclude any bipartite graph from the
graphon, then $J$ must be compact and finite dimensional.}

In \cite{LSz3} it was shown that weak regularity partitions of a
graphon $(J,W)$ (which generalize weak regularity partitions of
graphs in a natural way) correspond to covering $J$ with sets of
small diameter. In Section \ref{REGPART} we give a stronger and
cleaner version of this result. Combined with the results in Section
\ref{THIN}, we obtain the following fact: If a graph does not contain
a fixed bipartite graph $F$ as an induced sub-bigraph, then it has
polynomial size strong regularity partitions (in the error bound
$\eps$).

A motivation for our paper comes from extremal combinatorics. In
\cite{LSz6} we study the structure of graphons that arise as unique
solutions of extremal problems involving the densities of finitely
many subgraphs (we call such graphons {\it finitely forcible}). Such
graphons come up naturally in extremal graph theory. Quite
interestingly, all the examples of finitely forcible graphons
produced in \cite{LSz6} have a compact and finite dimensional
underlying metric space. The question arises wether every extremal
problem (involving a finite number of subgraph densities) has a
solution of this type.

Finally we mention that graph limit theory has a close connection to
the theory of dynamical systems. Probability spaces with measure
preserving actions can often be endowed by a natural topology in
which the action is continuous. The corresponding theory is called
topological dynamics. Informally speaking, we can say that the
relationship between graphons and topological graphons is similar to
the relationship between dynamics and topological dynamics.

\section{Preliminaries}\label{PRELIM}

We make a technical but useful distinction between bipartite graphs
and bigraphs. A {\it bipartite graph} is a graph $(V,E)$ whose node
set has a partition into two classes such that all edges connect
nodes in different classes. A {\it bigraph} is a triple $(U_1,U_2,E)$
where $U_1$ and $U_2$ are finite sets and $E\subseteq U_1\times U_2$.
So a bipartite graph becomes a bigraph if we fix a bipartition and
specify which bipartition class is first and second. On the other
hand, if $F=(V,E)$ is a graph, then $(V,V,E')$ is an associated
bigraph, where $E'=\{(x,y):~xy\in E\}$. This bigraph is obtained from
$F$ by a standard construction of doubling the nodes.

If $G=(V,E)$ is a graph, then an {\it induced sub-bigraph} of $G$ is
determined by two subsets $S,T\subseteq V$, and its edge set consists
of those pairs $(x,y)\in S\times T$ for which $xy\in E$ (so this is
an induced subgraph of the bigraph associated with $G$).

Let $J_i=(\Omega_i,\AA_i,\pi_i)$ ($i=1,2$) be (standard) probability
spaces. A measurable function $W:~J_1\times J_2 \to [0,1]$ is called
a {\it bigraphon}. A {\it graphon} is a special bigraphon where
$J_1=J_2=J$ and $W$ is symmetric: $W(x,y)=W(y,x)$ for all $x,y\in J$.

For a fixed probability space $J$, graphons can be considered as
elements of the space $L_\infty(J\times J)$. The norm that it most
important in their is study is, however, not the $L_\infty$ norm, but
the {\it cut-norm}, defined by
\[
\|W\|_\square=\sup_{S,T\subseteq J}\Bigl|\int\limits_{S\times T}
W(x,y)\,dx\,dy\Bigr|.
\]
We will also use the $L_1$ norm
\[
\|W\|_1=\int\limits_{J\times J} |W(x,y)|\,dx\,dy.
\]

A graphon $(J,W)$ is called a {\it stepfunction}, if there is a
partition of $J$ into a finite number of measurable sets
$S_1,\dots,S_n$ so that $W$ is constant on every $S_i\times S_j$. The
partition classes will be called the {\it steps} of the stepfunction.

Every graph $F=(V,E)$ can be considered as a graphon, if we consider
$V$ as a finite probability space with the uniform measure, and $E$,
as the indicator function of adjacency. We can resolve the atoms into
intervals of length $1/|V|$, to get a graphon $([0,1],W_F)$ (which is
a stepfunction). More explicitly, we split $[0,1]$ in $|V|$ equal
intervals $L_i$, and define $W_F(x,y)=E(i,j)$ for $i x\in L_i$ and
$y\in L_j$. This graphon is weakly isomorphic to $(V,E)$ (see below).

In a similar way, every bigraph can be considered as a finite
bigraphon, and defines a bigraphon $([0,1],[0,1],W_F)$.

\begin{remark}\label{REM:DIGRAPH}
We could consider the version of this notion where $J_1=J_2$ but $W$
is not necessarily symmetric. Such a structure arises as the limit
object of a convergent sequence of directed graphs with no parallel
edges, and therefore can be called a {\it digraphon}. We do not need
them in this paper.
\end{remark}

Every bigraphon $(J_1,J_2,W)$ can be considered as a linear kernel
operator $L_1(J_1)\to L_\infty(J_2)$, defined by
\[
f\mapsto \int\limits_J W(.,y)f(y)\,dy.
\]
Of course, this operator reamin well-defined if we increase the
subscript in $L_1$ in the domain and lower the subscript in
$L_\infty$ in the range. In the case of a graphon $(J,W)$, it is
useful to consider it as an operator $L_2(J)\to L_2(J)$, since it is
then a Hilbert-Schmidt operator, and a rich theory is applicable. In
particular, we know that it has a discrete spectrum.

If $(J_1,J_2,U)$ and $(J_2,J_3,W)$ are two bigraphons, we can define
their {\it operator product} $(J_1,J_3,U\circ W)$ by
\[
(U\circ W)(x,y)=\int\limits_{J_2} U(x,z)W(z,y)\,dz.
\]
(We will write $dz$ instead of $d\pi_2(z)$, where $\pi_2$ is the
measure on $J_2$: integrating over $J_2$ means that we integrate with
respect to the probability measure of $J_2$.)

The notion of the density of a graph in a graphon has been introduced
in \cite{FLS}. Here we need several versions, which unfortunately
leads to some messy notation. For a graphon $(J,W)$ and graph
$F=(V,E)$, we associate a variable $x_v\in J$ with every node $v\in
V$, and define
\[
t(F,W;x)= \prod_{uv\in E(F)} W(x_u,x_v), \qquad
t(F,W)=\int\limits_{J^V} t(F,W;x)\,dx.
\]
We can think of $t(F,W)$ as ``counting subgraphs isomorphic to $F$''.
We also need the induced version:
\begin{align*}
t_{\ind}(F,W;x)&= \prod_{uv\in E(F)} W(x_u,x_v)\prod_{u,v\in V\atop
uv\notin E(F)}
(1-W(x_u,x_v)) \\
t_\ind(F,W)&=\int\limits_{J^V} t_{\ind}(F,W;x)\,dx.
\end{align*}
For any subset $S\subseteq V$, we define $t_S(F,W;.):~J^S\to\R$ by
integrating only over variables corresponding to $V\setminus S$: If
$x'$ and $x''$ denote the restrictions of $x\in J^V$ to $S$ and
$V\setminus S$, respectively, then
\[
t_S(F,W; x')=\int\limits_{J^{V\setminus S}} t(F,W;x)\, dx''.
\]
Note that $t_\emptyset(F,W)=t(F,W)$ and $t_V(F,W;.)=t(F,W;.)$.

These quantities have obvious analogues for bigraphs and bigraphons.
For a bigraphon $(J_1,J_2,W)$ and bipartite graph $(U_1,U_2,E)$, we
introduce variables $x_u\in J_1$ $(u\in U_1)$ and $y_v\in J_2$ $(v\in
U_2)$, and define
\[
\tbp(F,W;x,y)=\prod_{uv\in E(F)} W(x_u,y_v), \qquad
\tbp(F,W)=\int\limits_{J_1^{U_1}}\int\limits_{J_2^{U_2}}\tbp(F,W;x,y)
\,dy\,dx.
\]
Again, we define an induced version:
\begin{align*}
\tbp_\ind (F,W;x,y) &= \prod_{ij\in E(F)} W(x_i,y_j)\prod_{i\in U_1,
j\in U_2\atop ij\notin E(F)} (1-W(x_i,y_j))\\ \tbp_\ind (F,W)
&=\int\limits_{J_1^{U_1}}\int\limits_{J_2^{U_2}} \tbp_\ind (F,W;x,y)
\,dy\,dx.
\end{align*}
Assume that subsets $S_i\subseteq U_i$ are specified. We define the
function $\tbp(F,W;.):~J_1^{S_1}\times J_2^{S_2}\to\R$ by
\[
\tbp_{S_1,S_2}(F,W; x',y')=\int\limits_{J_1^{U_1\setminus S_1}}
\int\limits_{J_2^{U_2\setminus S_2}}\tbp(F,W;x,y)\,dy''\,dx'',
\]
where, similarly as above, $x'$ and $x''$ denote the restrictions of
$x\in J_1^{U_1}$ to $S_1$ and $U_1\setminus S_1$, respectively, and
similarly for $y$. We can define $\tbp_{\ind;S_1,S_2}(F,W)(x',y')$
analogously.

Two graphons $(J,W)$ and $(J',W')$ are {\it weakly isomorphic} if for
every graph $F$, $t(F,W)=t(F,W')$. Various characterizations of weak
isomorphism were given in \cite{BCL}. Every graphon is weakly
isomorphic to a graphon on $[0,1]$ (with the Lebesgue measure), and
also to a (possibly different) graphon which is twin-free in the
sense that $W(x,.)$ and $W(x',.)$ differ on a set of positive measure
for all $x\not= x'$.

\section{The topology of graphons}\label{TOPGRAPH}

\subsection{The neighborhood distance}

Let $(J,W)$ be a graphon. We can endow the space $J$ with a distance
function by
\begin{align*}
r_W(x,y)=\|W(x,.)-W(y,.)\|_1.
\end{align*}
This function is defined for almost all pairs $x,y$; we can delete
those points from $J$ where $W(x,.)\notin L_1(W)$ (a set of measure
$0$), to have $r_W$ defined on all pairs. It is clear that $r_W$ is a
pre-metric (it is symmetric and satisfies the triangle inequality).
We call $r_W$ the {\it neighborhood distance} on $W$.

We also define metrics on bigraphons, endowing the spaces $J_1$ and
$J_2$ with distance functions by
\begin{align*}
r_1(x,y)&=\|W(x,.)-W(y,.)\|_1\qquad (x,y\in J_1),\\
r_2(x,y)&=\|W(.,x)-W(.,y)\|_1\qquad (x,y\in J_2).
\end{align*}
These functions are defined for almost all pairs $x,y$.

\begin{example}\label{EXA:SPHERE}
Let $S^k$ denote the unit sphere in $\R^{k+1}$, consider the uniform
probability measure on it, and let $W(x,y)=1$ if $x\cdot y\ge 0$ and
$W(x,y)=0$ otherwise. Then $(S^k,W)$ is a graphon, in which the
neighborhood distance of two points $a,b\in S^k$ is just their
spherical distance (normalized by dividing by $\pi$). Furthermore,
$1-2(W\circ W)(x,y)$ is just the spherical distance of $x$ and $y$,
and from here is is easy to see that the similarity distance is
within constant factors of the neighborhood distance.
\end{example}

\begin{example}\label{EXA:DIST}
Let $(M,d)$ be a metric space, and let $\pi$ be a Borel probability
measure on $M$. Assume that the diameter of $M$ is at most $1$. Then
$d$ can be viewed as a graphon on $(M,d)$. For $x,y\in M$, we have
\[
r_d(x,y)=\int\limits_M |d(x,z)-d(y,z)|\,d\pi(z) \le \int\limits_M
d(x,y) \,d\pi(z) = d(x,y),
\]
so the identity map $(M,d)\to (M,r_d)$ is contractive. This implies
that if $(M,d)$ is compact, and/or finite dimensional (in many senses
of dimension), then so is $(M,r_d)$. For most "everyday" metric
spaces like (like segments, spheres, or balls) $r_d(x,y)$ can be
bounded from below by $\Omega(d(x,y))$, in which case $(M,d)$ and
$(M,r_d)$ are homeomorphic.

More generally, if $F:~[0,1]\to[0,1]$ is a continuous function, then
$W(x,y)=F(d(x,y))$ defines a graphon, and the identity map
$(M,d)\to(M,r_W)$ is continuous.
\end{example}

\begin{example}\label{EXA:FINFORCE}
Finitely forcible graphons, mentioned in the introduction, give
interesting examples, for whose details we refer to \cite{LSz6}. One
class is stepfuctions (equivalent to finite weighted graphs), which
were proved to be finitely forcible by Lov\'asz and S\'os
\cite{LSos}; for these, the underlying metric space is finite. Other
examples introduced in \cite{LSz6} provide as underlying topologies
an interval, the Cantor set, and the one-point compactification of
$\N$.
\end{example}

\subsection{Pure [bi]graphons}

A bigraphon $(J_1,J_2,W)$ is {\it pure} if $(J_i, r_i)$ is a complete
separable metric space and the probability measure has full support
(i.e., every open set has positive measure). This definition includes
that $r_i(x,y)$ is defined for all $x,y\in J_i$ and $r_i(x,y)>0$ if
$x\not=y$, i.e., the bigraphon has no "twin points". We say that a
graphon is {\it pure}, if the underlying metric probability space is
complete, separable and the probability measure has full support.

\begin{theorem}\label{THM:PURE}
Every [bi]graphon is weakly isomorphic to a pure [bi]graphon.
\end{theorem}

\begin{remark}\label{REDUCE}
It was shown in \cite{BCL} that every graphon is weakly isomorphic to
a graphon on a standard probability space with no parallel points,
which means that for any two points $x,x'\in J$, $W(x,.)$ and
$W(x',.)$ differ on a set of positive measure. Lemma \ref{LEM:TYP}
can be considered as a strengthening of this result.
\end{remark}

\begin{proof}
We give the proof for bigraphons; the case of graphons is similar. We
assume that $J_1$ and $J_2$ are standard probability spaces; this can
be achieved similarly as for graphons. Let $T_1$ be the set of
functions $f\in L_1[J_2]$ such that for every $L_1$-neighborhood $U$
of $f$, the set $\{x\in J_1:~W(x,.)\in U\}$ has positive measure.

\begin{claim}\label{CLAIM:TYP}
For almost every point $x\in J_1$, $W(x,.)\in T_1$.
\end{claim}

Indeed, it is clear that for almost all $x\in J_1$, $W(x,.)\in
L_1[J_2]$. Every function $g \in L_1[J_2]\setminus T_1$ has an open
neighborhood $U_g$ in $L_1[J_2]$ such that $\pi_1\{x\in
J_1:~W(x,.)\in U_g\}=0$. Let $U = \bigcup_{g\notin T_1} U_g$. Since
$L_1[J_2]$ is separable, $U$ equals the union of some countable
subfamily $\{U_{g_i}:~i\in\N\}$ and thus $\pi_1\{x\in J_1:~W(x,.)\in
U\}=0$. Since if $W(x,.)\notin T_1$ then $W(x,.) \in U$, this proves
the Claim.

Clearly $T_1$ inherits a metric from $L_1[J_2]$, and it is complete
and separable in this metric. The functions $W(x,.)$ are everywhere
dense in $T_1(W)$ and have measure $1$. It also inherits a
probability measure $\pi'_1$ from $J_1$ through
\[
\pi'_1(X)= \pi_1\{x\in\Omega_1:~W(x,.)\in X\}.
\]
So $T_1$ is a complete separable metric space with a probability
measure on its Borel sets. It also follows from the definition of
$T_1$ that every open set has positive measure.

Define $\widetilde{W}:~T_1\times J_2\to[0,1]$ by
$\widetilde{W}(f,y)=f(y)$ for $f\in T_1$ and $y\in J_2$. Then we can
replace $J_1$ by $T_1$ and $W$ by $\widetilde{W}$, to get a weakly
isomorphic graphon. Similarly, we can replace $J_2$ by $T_2$.
\end{proof}

We say that two graphons $(J,W)$ and $(J',W')$ are {\it isometric} if
there is an isometric bijection $\phi:~J\to J'$ that is measure
preserving, and $W'(\phi(x),\phi(y))=W(x,y)$ for almost all $x,y\in
J$. The definition for bigraphons is slightly more complicated: two
bigraphons $(J_1,J_2,W)$ and $(J_1',J_2',W')$ are {\it isometric} if
there are isometric, measure preserving bijections $\phi_1:~J_1\to
J'_1$ and $\phi_2:~J_2\to J'_2$ such that
$W'(\phi_1(x),\phi_2(y))=W(x,y)$ for almost all $(x,y)\in J_1\times
J_2$.

\begin{theorem}\label{LEM:TYP}
If two pure [bi]graphons are weakly isomorphic, then they are
isometric.
\end{theorem}

\begin{proof}
We describe the proof for graphons. Theorem 2.1 (a) in \cite{BCL}
says that if two graphons $(J,W)$ and $(J',W')$ are weakly
isomorphic, and they have no twins, then one can delete delete
$0$-sets $S\subseteq J$ and $S'\subseteq J'$ such that there is a
bijective measure preserving map $\phi:~J\setminus S\to J'\setminus
S'$ such that $W'(\phi(x),\phi(y))=W(x,y)$ for almost all $(x,y)\in
J\times J$. We may even assume that for every $x\in J\setminus S$,
$W'(\phi(x),\phi(y))=W(x,y)$ holds for almost all $y$ (and vice
versa), since this can be achieved by deleting further $0$-sets.
Clearly $\phi$ preserves the metric.

We also know that $J\setminus S$ is dense in $J$ (since $(J,W)$ is
pure and so its probability measure has full support), and so $J$ is
the completion of $J\setminus S$ (and similarly for $J'$). Hence
$\phi$ extends to an isometry between $J$ and $J'$, which shows that
$(J,W)$ and $(J',W')$ are isometric graphons.
\end{proof}

\begin{remark}\label{NORMALIZE}
Is purity the ultimate normalization of a graphon? There is still
some freedom left: we can change the value of $W$ on a symmetric
subset of $J\times J$ that intersects every fiber $J\times \{v\}$ in
a set of measure. We can take the integral of $W$ (which is a measure
$\omega$ on $J$), and then the derivative of $\omega$ wherever this
exists. This way we get back $W$ almost everywhere, and a well
defined value for some further points. What is left undefined is the
set of ``essential discontinuity'' of $W$ (of measure $0$). It would
be interesting to relate this set to combinatorial properties of $W$.
\end{remark}

\subsection{Density functions on pure [bi]graphons}

The following technical Lemma will be very useful in the study of
$r_W$ and related distance functions.

\begin{lemma}\label{LEM:CONTIN}
{\rm(a)} Let $(J,W)$ be a graphon, $F$, a graph, and $S\subseteq V$,
an independent set of nodes. Then the function
$t=t_S(F,W;.):~J^S\to\R$ satisfies
\[
|t(x)-t(x')|\le |E| \max_{i\in S} r_W(x_i,x_i').
\]

\smallskip

{\rm(b)} Let $(J_1,J_2,W)$ be a bigraphon, let $F=(U_1,U_2,E)$ be a
bigraph, and let $S_i\subseteq U_i$ be such that no edge connects
$S_1$ to $S_2$. Then the function $t=\tbp_{S_1,S_2}(F,W,.):
J_1^{S_1}\times J_2^{S_2}\to\R$ satisfies
\[
|t(x,y)-t(x',y')|\le |E| \max\{\max_{i \in S_1}
r_1(x_i,x_i'),\max_{j\in S_2} r_2(y_j,y_j')\}.
\]
\end{lemma}

\begin{remark}\label{REM:LIP}
(i) It follows that the functions $t$ in (a) and (b) are Lipschitz
(and hence continuous).

\smallskip

(ii) In both parts (a) and (b) of the Lemma, the graph $F$ could have
multiple edges.
\end{remark}

\begin{proof}
We describe the proof of (a); the proof of (b) is similar. For each
$i\in U\setminus S$, let $x_i=x_i'$ be a variable. Let
$E=\{u_1v_1,\dots u_mv_m\}$, where we may assume that $v_i\in
U\setminus S$. Then
\begin{align*}
&t(x)-t(x')= \int\limits_{J^{U\setminus S}} \prod_{i=1}^m
W(x_{u_i},x_{v_i})\,dy - \int\limits_{J^{U\setminus S}}
\prod_{i=1}^m W(x'_{u_i},x'_{v_i})\,dy\\
&~~=\sum_{j=1}^m\int\limits_{J^{U\setminus S}}
\prod_{i<j}W(x_{u_i},x_{v_i})
(W(x_{u_j},x_{v_j})-W(x'_{u_j},x'_{v_j}))
\prod_{j>i}W(x'_{u_i},x'_{v_i}),dy
\end{align*}
and hence
\[
|t(x)-t(x')|\le \sum_{j=1}^m\int\limits_{J^{U\setminus S}}
|W(x_{u_j},x_{v_j})-W(x'_{u_j},x'_{v_j})|\,dy.
\]
By the assumption that $v_i\in U\setminus S$, we have
$x_{v_j}=x'_{v_j}$ for every $j$, and so
\[
|t(x)-t(x')| \le \sum_{j=1}^m r_W(x_{u_j},x'_{u_j}) \le |E|
\max_{1\le i \le k} r_W(x_i,x_i'),
\]
which proves the assertion.
\end{proof}

Lemma \ref{LEM:CONTIN} has an important corollaries for pure
graphons, which are closely related to Lemma 2.8 in \cite{LSz6}. We
do not formulate all versions, just a few that we need.

\begin{corollary}\label{COR:CONTIN}
Let $(J,W)$ be a pure graphon, and let $F$ be a graph and let
$S\subseteq V$, where $S$ is independent. Then $t_S(F,W;x)$ is a
continuous function of $x\in J^S$.
\end{corollary}

Applying this when $F$ is a path of length 2, we get:

\begin{corollary}\label{COR:SQUARE}
For every pure graphon $(J,W)$, $W\circ W$ is a continuous function
on $J$.
\end{corollary}

Another application of Corollary \ref{COR:CONTIN} gives:

\begin{corollary}\label{COR:DERAND1}
Let $(J,W)$ be a pure graphon, and let $F_1,\dots,F_m$ be graphs
whose node set contains a common set $S$, which is independent in
each. Let $T\subseteq S$, and let $a_1,\dots,a_m$ be real numbers.
Let $x\in J^T$, and assume that the equation
\begin{equation}\label{EQ:1}
\sum_{i=1}^m a_i t_S(F_i,W;x,y) =0
\end{equation}
holds for almost all $y\in J^{S\setminus T}$. Then it holds for all
$y\in J^{S\setminus T}$.
\end{corollary}

\begin{proof}
By Corollary \ref{COR:CONTIN}, the left hand side of \eqref{EQ:1} is
a continuous function of $(x,y)$, and so it remains a continuous
function of $y$ if we fix $x$. Hence the set where it is not $0$ is
an open subset of $J^{S\setminus T}$. Since the graphon is pure, it
follows that this set is either empty of has positive measure.
\end{proof}

We formulate one similar corollary for bigraphons.

\begin{corollary}\label{COR:DERAND}
Let $(J_1,J_2,W)$ be a pure bigraphon, and let $F_1,\dots,F_m$ be
bigraphs with the same bipartition classes $U_1$ and $U_2$. Let
$a_1,\dots,a_m$ be real numbers. Assume that the equation
\begin{equation}\label{EQ:2}
\sum_{i=1}^m a_i\tbp_{U_1}(F_i,W;x)=0
\end{equation}
holds for almost all $x\in J_1^{U_1}$. Then it holds for all $x\in
J_1^{U_1}$.
\end{corollary}

\subsection{The similarity distance}

It turns out (it was already noted in \cite{LSz3}) that the distance
function $r_{W\circ W}$ defined by the operator square of $W$ is also
closely related to combinatorial properties of a graphon. We call
this the {\it similarity distance} (for reasons that will become
clear later). In explicit terms, we have
\begin{align}\label{EQ:SIM-DEF}
r_{W\circ W}(a,b)&= \int\limits_J\Bigl| \int\limits_J
W(a,y)W(y,x)\,dy- \int\limits_J
W(b,y)W(y,x)\,dy\Bigr|\,dx\nonumber\\
&= \int\limits_J\Bigl|\int\limits_J
W(x,y)\bigl(W(y,a)-W(y,b)\bigr)\,dy\Bigr|\,dx.
\end{align}

\begin{remark}\label{SAMPLE}
Let $\Xb,\Yb,\Zb$ be independent uniform random points from $J$, then
we can rewrite the definitions of these distances as
\begin{align}
r_W(a,b)&= \E_\Xb |W(\Xb,a)-W(\Xb,b)|,\label{EQ:NEIGH-DEF1}\\
r_{W\circ W}(a,b)&=
\E_\Xb\bigl|\E_\Yb(W(\Xb,\Yb)(W(\Yb,a)-W(\Yb,b)))\bigr|.\label{EQ:SIM-DEF1}
\end{align}
This formulation shows that this distance can be computed with
arbitrary precision from a bounded size sample. We do not go into the
details of this.
\end{remark}

\begin{lemma}\label{LEM:HAUSDORFF}
If $(J,W)$ is a pure graphon, then the similarity distance $r_{W\circ
W}$ is a metric.
\end{lemma}

\noindent So $(J,r_{W\circ W})$ is a metric space, and hence
Huasdorff. We will show later that it is always compact.

\begin{proof}
The only nontrivial part of this lemma is that $r_{W\circ W}(x,y)=0$
implies that $x=y$. The condition $r_{W\circ W}(x,y)=0$ implies that
for almost all $u\in J$ we have $(W\circ W)(x,u)=(W\circ W)(y,u)$, or
more explicitly
\[
\int\limits_J (W(x,z)-W(y,z))W(z,u)\,dz =0.
\]
Using that $(J,W)$ is pure, Corollary \ref{COR:DERAND} implies that
this holds for every $u\in J$. in particular, it holds for $u=x$ and
$u=y$. Taking the difference, we get that
\[
\int\limits_J (W(x,z)-W(y,z))(W(z,x)-W(z,y))\,dz =0,
\]
and hence $W(x,z)=W(y,z)$ almost everywhere. Using again that $(J,W)$
is pure, we get that $x=y$.
\end{proof}

For every $x\in J$, the function $W(x,.)$ is in $L_\infty(J)$, and
hence the weak topology of $L_1(J)$ gives a topology on $J$. It is
well known that when restricted to $L_\infty(J)$, this topology is
the weak-$*$ topology on $L_\infty(J)$, and hence it is metrizable,
and the unit ball of $L_\infty(J)$ is compact in it (Alaoglu's
Theorem). A sequence of points $(x_n)$ is convergent in this topology
if and only if
\[
\int_A W(x_n,y)\,dy\rightarrow \int_A W(x,y)\,dy
\]
for every measurable set $A\subseteq J$. We call this the {\it weak
topology} on $J$. We need this name only temporarily, since we are
going to show that $r_{W\circ W}$ gives a metrization of the weak
topology.

\begin{theorem}\label{THM:WEAKWW}
For any pure graphon, the metric $r_{W\circ W}$ defines exactly the
weak topology.
\end{theorem}

\begin{proof}
First we show that the weak topology is finer than the topology of
$(J, r_{W\circ W})$. Suppose that $x_n\to x$ in the weak topology,
and consider
\[
r_{W\circ W}(x_n,x) = \int\limits_J\Bigl|\int\limits_J
\bigl(W(x_n,y)-W(x,y)\bigr)W(y,z)\,dy\Bigr|\,dz.
\]
Here the inner integral tends to $0$ for every $z$, by the weak
convergence $x_n\to x$. Since it also remains bounded, it follows
that the outer integral tends to $0$. This implies that $x_n\to x$ in
$(J, r_{W\circ W})$.

From here, the equality of the two topologies follows by general
arguments: the weak topology is compact, and the coarser topology of
$r_{W\circ W}$ is Hausdorff, which implies that they are the same.
\end{proof}

\begin{corollary}\label{COR:WWCOMP}
For every pure graphon $(J,W)$, the space $(J, r_{W\circ W})$ is
compact.
\end{corollary}

To compare the topology of $(J,r_W)$ with these, note that for any
two points $x,y\in J$, we have
\begin{equation}\label{EQ:CONTRACT}
r_{W\circ W}(x,y) \le r_W(x,y),
\end{equation}
which implies that the topology of $(J, r_W)$ is finer than the
topology of $(J, r_{W\circ W})$.

\subsection{Compact Graphons}\label{SEC:COMPACT}

Graphons for which the finer space $(J,r_W)$ is also compact seem to
have a special importance in combinatorics. Let us call such a
graphon a {\bf compact graphon}.

\begin{prop}\label{THM:COMPACT}
A pure graphon $(J,W)$ is compact if and only if $(J,r_W)$ and
$(J,r_{W\circ W})$ define the same topologies.
\end{prop}

\begin{proof}
If the topologies $(J,r_W)$ and $(J,r_{W\circ W})$ are the same, then
$(J,r_W)$ is compact by Corollary \ref{COR:WWCOMP}. Conversely, if
$(J,r_W)$ is compact then, by the argument used before in the proof
of Theorem \ref{THM:WEAKWW}, the coarser Hausdorff topology of
$(J,r_{W\circ W})$ must be the same.
\end{proof}

\begin{example}\label{EXA:SYM}
Let $J=[0,1]$, $f(y)=\lfloor\log(1/y)\rfloor$, and define
\[
W(x,y)=
  \begin{cases}
    x_{f(y)}, & \text{if $x>1/2$ and $y\le 1/2$}, \\
    y_{f(x)}, & \text{if $x\le 1/2$ and $y>1/2$},\\
    0, & \text{otherwise},
  \end{cases}
\]
where $x=0.x_1x_2\dots$ and $y=0.y_1y_2\dots$ are the binary
expansions of $x$ and $y$, respectively. Then selecting one point
from each interval $[2^{-{k+1}},2^{-(k)}]$, we get an infinite number
of points in $([0,1],r_2)$ mutually at distance $1/4$, so $(J,W_r)$
is not compact, but by Corollary \ref{COR:WWCOMP}, $(J,r_{W\circ W})$
is compact. So the two topologies are different.
\end{example}

We conclude this section with an observation relating the topology of
$J$ to spectral theory.

\begin{lemma}\label{EIGEN}
Let $(J,W)$ be a pure graphon. Then every eigenfunction $f\in L_2(J)$
of $W$ as a kernel operator belonging to a nonzero eigenvalue is
continuous in the metric $r_{W\circ W}$ (and therefore also in
$r_W$).
\end{lemma}

\begin{proof}
It suffices to prove that $f$ is continuous in $(J,r_W)$, since we
can apply the argument to the graphon $(J,W\circ W)$, which also has
$f$ as an eigenvector.

First, we have
\[
|f(x)|=\frac1{|\lambda|}\left|\int_J W(x,y)f(y)\,dy\right|\le
\frac1{|\lambda|}\|f\|_1\le \frac1{|\lambda|}\|f\|_2,
\]
and so $f$ is bounded. We know by Corollary \ref{COR:SQUARE} that
$W\circ W$ is continuous in $(J,r_W)$, and hence so is
\[
f=\frac1{\lambda^2}\int\limits_J (W\circ W)(x,y)f(y)\,dy.
\]
\end{proof}

\section{Thin graphons}\label{THIN}

\subsection{The main theorem}

We say that a bigraphon $W$ is {\it thin} if there is a bigraph $F$
such that $\tbp_\ind(F,W)=0$. Trivially, if $W$ is thin, then so is
its complementary bigraphon $1-W$.

We call a graphon {\it thin} if it is thin as a bigraphon. (Note: for
this, it is not enough to require $t_\ind(F,W)=0$ for some bipartite
graph $F$. For example, consider the graphon $U:~[0,1]^2\to[0,1]$
defined by $U(x,y)=U(y,x)=1/2$ if $x\in[0,1/2]$ and $y\in(1/2,1]$,
and $U(x,y)=1$ otherwise. As a bigraphon, this is not thin, but
satisfies $t_\ind(F,W)=0$ for every bigraph with at least 3 nodes in
one of the classes.

The {\it (upper) packing dimension} of a metric space $(M,d)$ is
defined as
\[
\limsup_{\eps\to 0}\frac{\log N(\eps)}{\log(1/\eps)},
\]
where $N(\eps)$ is the maximum number of points in $M$ mutually at
distance at least $\eps$. So this dimension is finite if and only if
there is a $d\ge0$ such that every set of points mutually at distance
at least $\eps$ has at most $\eps^{-d}$ elements. It is easy to see
that we could use instead of $N(\eps)$ the minimum number of sets of
diameter at most $\eps$ covering the space.

Our main goal is to prove:

\begin{theorem}\label{THM:FINDIM}
If a pure bigraphon $(J_1,J_2,W)$ is thin, then {\rm(a)}
$W(x,y)\in\{0,1\}$ almost everywhere, {\rm(b)} $J_1,J_2$ are compact,
and {\rm(c)} $J_1,J_2$ have finite packing dimension.
\end{theorem}

\begin{remark}\label{REM:FINDIM}
The proof will show that if $t_\ind(F,W)=0$ for a bigraph $F$ with
$k$ nodes, then the packing dimension of $J_i$ is bounded by $10|F|$.
\end{remark}

Before giving the proof, we describe a class of examples, and then
recall some facts about the Vapnik-\v Cervonenkis dimension.

\begin{example}\label{EXA:ASYM}
Let $V$ be a finite or countable set, $\pi$, a probability measure on
$V$, and define $J_1=[0,1]^V$, $J_2=[0,1]\times V$, with the power
measure $\mu_1$ on $J_1$ and the product measure $\mu_2$ on $J_2$. We
define a bigraphon on $J_1\times J_2$ by
\[
W(x,y)=\one_{t\le x_i}
\]
for $x=(x_i:~i\in S)$ and $y=(t,i)$. We can metrize this bigraphon by
\[
r_1(x,x')=\sum_{i\in V} \pi(i)|x_i-x_i'|
\]
for $x=(x_i:~i\in S),~x'=(x'_i:~i\in S)\in J_1$, and
\[
r_2(y,y')=
  \begin{cases}
    |t-t'| & \text{if $i=1'$}, \\
    t+t'-2tt' & \text{otherwise}.
  \end{cases}
\]
for $y=(t,i),~y'=(t',i')\in J_2$.

If $V$ is finite, then $(J_1,r_1)$ has dimension $|V|$, while
$(J_2,r_2)$ has dimension $1$, and both are compact. These facts also
follow if we observe that $W$ is thin. Indeed, if $F$ denotes the
matching with $|V|+1$ edges, then $\tbp_\ind(F,W)=0$, since among any
$|V|+1$ points in $J_2$, there are two points of the form $y=(t,i)$
and $y'=(t',i)$ with $t<t'$, and then $W(.,(t,i))\ge W(.,(t',i))$.

If $V$ is infinite, then $(J_1,r_1)$ is infinite dimensional but
compact, while $(J_2,r_2)$ is not compact.
\end{example}

\begin{example}\label{EXA:ASYM2}
Let $J_1=J_2=[0,1]$, and let $W(x,y)=x_{f(y)}$, where
$x=0.x_1x_2\dots$ is the binary expansion of $x$, and
$f(y)=\lceil\log(1/y)\rceil$. Then for $x=0.x_1x_2\dots$ and
$x'=0.x'_1x'_2\dots$ we have $r_1(x,x')=\sum_{k=1}^\infty
2^{-k}|x_k-x'_k|$, and from here is is easy to see that $([0,1],r_1)$
is compact. Furthermore, if $S\subseteq[0,1]$ is a set of points
mutually more than $2^{-n}$ apart, then any two elements of $S$ must
differ in one of their first $n$ digits, and so their number is at
most $2^n$. Hence the packing dimension of $([0,1],r_1)$ is $1$.

On the other hand, selecting a point $y_k\in[2^{-k},2^{-(k-1)}]$, we
get an infinite number of points in  $([0,1],r_2)$ mutually at
distance $1/2$, so this space is not compact and infinite
dimensional.
\end{example}

\subsection{Vapnik-\v{C}ervonenkis dimension}

For any set $V$ and family of subsets $\HH\subseteq 2^V$, a set
$S\subseteq V$ is called {\it shattered}, if for every $X\subseteq S$
there is a $Y\in\HH$ such that $X=Y\cap S$. The {\it
Vapnik-\v{C}ervonenkis dimension} or {\it VC-dimension} $\dim_{\rm
VC}(\HH)$ of a family  of sets is the supremum of cardinalities of
shattered sets \cite{VC}. For us, $k$ will be always finite.

Let $V$ be a probability space and $\HH$, a family of measurable
subsets of $V$. A finite subfamily $\HH'$ is {\it qualitatively
independent} if all the $2^{|\HH'|}$ atoms of the set algebra they
generate have positive measure. The {\it dual essential
Vapnik-\v{C}ervonenkis dimension}, or briefly {\it DE-dimension}, of
$\HH$ is a supremum of all cardinalities of qualitatively independent
subfamilies of $\HH$.

We recall two basic facts about VC-dimension:

\begin{lemma}[Sauer-Shelah Lemma]\label{LEM:SaSh}
If a family $\HH$ of subsets of an $m$-element set has VC-dimension
$k$, then
\[
|\HH|\le 1+m+\dots+\binom{m}{k}.
\]
\end{lemma}

For a family $\HH$ of sets, we denote by $\tau(\HH)$ the minimum
cardinality of a set meeting every member of $\HH$. The following
basic fact about VC-dimension was proved by Koml\'os, Pach and
Woeginger \cite{KPW}, based on the results of Vapnik and
\v{C}ervonenkis \cite{VC} (we do not state it in its sharpest form):

\begin{theorem}\label{LEM:TAU-STAR}
Let $J$ be a probability space and, $\HH$ a family of measurable
subsets of $J$ such that every $A\in\HH$ has measure at least $\eps$.
Suppose that $\HH$ has finite VC-dimension $k$. Then
\[
\tau(\HH)\le 8k\frac1\eps\log\frac1\eps.
\]
\end{theorem}

We need a couple of further facts. For a family $\HH$ of sets, let
$\HH(\triangle)\HH=\{A\triangle B:~A,B\in\HH\}$.

\begin{lemma}\label{LEM:VC-DIFF}
For every family of sets, $\dim_{\rm VC}(\HH(\triangle)\HH)\le
10\dim_{\rm VC}(\HH)$.
\end{lemma}

\begin{proof}
Set $k=\dim_{\rm VC}(\HH)$. Let $S$ be a subset of $V=\cup\HH$ with
$m$ elements that is shattered by $\HH(\triangle)\HH)$. Then every
$X\subseteq S$ arises as $X=(A\triangle B)\cap S$, where $A,B\in\HH$.
Since $(A\triangle B)\cap S= (A\cap S)\triangle(B\cap S)$, the number
of different sets of the form $A\cap S$ is at least $2^{m/2}$. By the
Sauer-Shelah Lemma, this implies that
\[
2^{m/2}\le 1+m+\dots+\binom{m}{k},
\]
whence $m\le 10k$ follows by standard calculation.
\end{proof}

\begin{lemma}\label{LEM:VC}
Let $\HH$ be a family of measurable sets in a probability space with
VC-dimension $k$ such that $\pi(A\triangle B)\ge\eps$ for all
$A,B\in\HH$. Then $|\HH|\le (80k)^k \eps^{-20k}$.
\end{lemma}

\begin{proof}
Consider the family $\HH'=\HH(\triangle)\HH$. Every $A\in\HH'$ has
$\pi(A)\ge 1/\eps$, and $\dim_{\rm VC}(\HH')\le 10k$ by Lemma
\ref{LEM:VC-DIFF}. Hence by Theorem \ref{LEM:TAU-STAR}, we have
\[
\tau(\HH')\le 80 k \frac1\eps\ln\frac1\eps.
\]
Let $S\subseteq \cup \HH$ be a set of size $\tau(\HH')$ meeting every
symmetric difference $A\triangle B$ ($A,B\in\HH$). Then the sets
$S\cap A$, $A\in\HH$ are all different. By the Sauer-Shelah Lemma,
this implies that
\[
|\HH|\le 1+|S|+\dots+\binom{|S|}{10k} < |S|^{10k} \le \left(80 k
\frac1\eps\ln\frac1\eps\right)^{10k}< (80k)^{10 k} \eps^{-20k}.
\]
\end{proof}

\subsection{VC-dimension and graphons}

\begin{lemma}\label{LEM:THIN}
Let $(J_1,J_2,W)$ be a pure $0$-$1$ valued bigraphon. Then $W$ is
thin if and only if the DE-dimension of the family
$\RR_W=\{\supp(W(x,.)):~x\in T_1\}$ is finite.
\end{lemma}

\begin{proof}
Suppose that this dimension is infinite. We claim that
$\tbp_\ind(F,W)>0$ for every bipartite graph $F=(U,U',E)$. Let
$S\subseteq J_1$ be a set such that the subfamily
$\{\supp(W(x,.)):~x\in T_1\}$ is qualitatively independent. To each
$i\in U$, assign a value $x_i\in S$ bijectively. By Corollary
\ref{COR:DERAND}, the set of points $y\in J_2$ such that
$\supp(W(.,y))\cap S = \{x_i:~i\in N(j)\}$ has positive measure for
each $j\in U'$. Hence $\tbp_\ind(F,W)>0$.

Conversely, suppose that $k=\dim(\RR_W)$ is finite. Let $F$ denote
the bipartite graph with $k+1$ nodes in one class $U$ and $2^{k+1}$
nodes in the other class $U'$, in which the nodes in $U'$ have all
different neighborhoods. Then $\tbp_\ind(F,W)=0$.
\end{proof}

\begin{remark}\label{RF-VC-QUANT}
The proof above in fact gives the following quantitative result:
$\tbp_\ind(F,W)=0$ for some bigraph $F$ with $k$ nodes in its smaller
bipartition class if and only if $\dim_{\rm DE}(\RR_W)<k$.
\end{remark}

\begin{proof*}{Theorem \ref{THM:FINDIM}}
We may assume that $W$ is pure.

(a) Suppose that the bigraph $F=(U_1,U_2,E)$ satisfies
$\tbp_\ind(F,W)=0$. Then for almost all $x\in J_1^{U_1}$, we have
$\tbp_{U_1,\ind}(F,W;x)=0$. By Corollary \ref{COR:DERAND}, it follows
that $\tbp_{U_1,\ind}(F,W;x)=0$ for every $x$. In particular,
$\tbp_{U_1,\ind}(F,W;x_0,\dots,x_0)=0$ for all $x_0\in J_1$. But for
this substitution,
\[
\tbp_{U_1,\ind}(F,W;x_0,\dots,x_0) = \int\limits_{J_2^{V_2}}
\prod_{j\in J_2} W(x_0,y_j)^{d_F(j)} (1-W(x_0,y_j))^{|U_1|-d_F(j)},
\]
and so for every $x_0$ we must have $W(x_0,y_0)\in\{0,1\}$ for almost
all $y_0$.

\smallskip

(b) By Theorem \ref{THM:COMPACT} it suffices to prove that if
$W(x_n,.)$, $n=1,2,\dots$ weakly converges to $f$, i.e.,
\[
\lim_{n\to\infty} \int\limits_S W(x_n,y)\,dy \to \int\limits_S
f(y)\,dy
\]
for every measurable set $S\subseteq J_2$, then it is also convergent
in $L_1$.

\begin{claim}\label{CLAIM:NON-RAND}
The weak limit function $f$ is almost everywhere $0$-$1$ valued.
\end{claim}

Suppose not, then there is an $\eps>0$ and a set $Y\subseteq J_2$
with positive measure such that $\eps\le f(x)\le 1-\eps$ for $x\in
Y$. Let $S_n=\supp(W(x_n,.))\cap Y$. We select, for every $k\ge 1$,
$k$ indices $n_1,\dots n_k$ so that the Boolean algebra generated by
$S_{n_1},\dots S_{n_k}$ has $2^k$ atoms of positive measure. If we
have this for some $k$, then for every atom $A$ of the boolean
algebra
\[
\lambda(A\cap S_n) = \int\limits_A W(x,y_n)\,dx \longrightarrow
\int\limits_A f(x)\,dx \qquad (n\to\infty),
\]
and so if $n$ is large enough then
\[
\frac{\eps}2 \lambda(A) \le \lambda(A\cap S_n) \le
\Bigl(1-\frac{\eps}2\Bigr) \lambda(A).
\]
If $n$ is large enough, then this holds for all atoms $A$, and so
$S_n$ cuts every previous atom into two sets with positive measure,
and we can choose $n_{k+1}=n$.

But this means that the DE-dimension of the supports of the $W(x,.)$
is infinite, contradicting Lemma \ref{LEM:THIN}. This proves Claim
\ref{CLAIM:NON-RAND}.

So we know that $f(x)\in\{0,1\}$ for almost all $x$,  and hence
\[
\|f-W(.,y_n)\|_1 = \int\limits_{\{f=1\}} (1-W(x,y_n))\,dx +
\int\limits_{\{f=0\}} W(x,y_n)\,dx \longrightarrow 0.
\]
Thus $W(.,y_n)\to f$ in $L_1$, which we wanted to prove.

\medskip

(c) Let $F=(U_1,U_2,E)$ be a bigraph such that $\tbp_\ind(F,W)=0$,
and let $U_i=[k_i]$. We show that the packing dimension of $J_1$ is
at most $10k_2$. To this end, we show that if any two elements of a
finite set $Z\subseteq J_1$ are at a distance at least $\eps$, then
$|Z|\le c(k) \eps^{-2k_2}$. Let $\HH=\{\supp(W(x,.)):~x\in Z\}$, then
\begin{equation}\label{EQ:TAVOL}
\pi_2(X\triangle Y)\ge \eps
\end{equation}
for any two distinct sets $X,Y\in\HH$.

Let $A$ be the union of all atoms of the set algebra generated by
$\HH$ that have measure $0$. Clearly $A$ itself has measure $0$, and
hence the family $\HH'=\{X\setminus A:~X\in\HH\}$ still has property
\eqref{EQ:TAVOL}.

We claim that $\HH'$ has VC-dimension less than $k_2$. Indeed,
suppose that $J_2\setminus A$ contains a shattered set $S$ with
$|S|=k_2$. To each $j\in U_2$, assign a point $q_j\in S$ bijectively.
To each $i\in U_1$, assign a point $p_i\in Z$ such that $q_j\in
\supp(W(p_i,.))$ if and only if $ij\in E$. This is possible since $S$
is shattered. Now fixing the $p_i$, for each $j$ there is a subset of
$J_2$ of positive measure whose points are contained in exactly the
same members of $\HH'$ as $q_j$, since $q_j\notin A$. This means that
the function $t=\tbp_{J_1,\ind}(F,W;.):~V_1^{J_1}\to \R$ satisfies
$t(p)>0$. Corollary \ref{COR:DERAND} implies that $t(x)>0$ for a
positive fraction of the choices of $x\in J_1^{V_1}$, and hence
$\tbp_\ind(F,W)>0$, a contradiction.

Applying Lemma \ref{LEM:VC} we conclude that $|Z|=|\HH|\le
(80k_2)^{10k_2} \eps^{-20k_2}$.
\end{proof*}

\subsection{Hereditary properties and thin bigraphons}

A graph property $\PP$ is a class of finite graphs closed under
isomorphism. The property is called {\it hereditary}, if whenever
$G\in\PP$, then every induced subgraph is also in $\PP$.

Let $\PP$ be any graph property. We denote by $\overline{\PP}$ its
{\it closure}, i.e., the class of graphons $(J,W)$ that arise as
limits of graph sequences in $\PP$. For every graphon $W$, let
$\II(W)$ denote the set of those graphs $F$ for which
$t_\ind(F,W)>0$. Clearly, $\II(W)$ is a hereditary graph property.

Let $\PP$ be a hereditary property of graphs. Then
\begin{equation}\label{EQ:PPII}
\cup_{W\in\overline{\PP}} \II(W) \subseteq\PP.
\end{equation}
Indeed, if $F\notin\PP$, then $t_\ind(F,G)=0$ for every $G\in\PP$,
since $\PP$ is hereditary. This implies that $t_\ind(F,W)=0$ for all
$W\in\overline{\PP}$, and so $F\notin \II(W)$.

Equality does not always hold in \eqref{EQ:PPII}. For example, we can
always add a bigraph $G$ and all its induced subgraphs to $\PP$
without changing $\overline{\PP}$. As a less trivial example,
consider all bigraphs with degrees bounded by 10. This property is
hereditary, and $\overline{\PP}$ consists of a single bigraphon (the
identically $0$ function).

\begin{prop}\label{PROP:NOFRILL}
For a hereditary property $\PP$ of graphs equality holds in
\eqref{EQ:PPII} if and only if for every graph $G\in\PP$ and $v\in
V(G)$, if we add a new node $v'$ and connect it to all neighbors of
$v$, then at least one of the two graphs obtained by joining or not
joining $v$ and $v'$ has property $\PP$.
\end{prop}

\begin{proof}
Suppose that this condition holds. Let $F\in\PP$ have $n$ nodes, and
let $F(k)$ denote a graph in $\PP$ obtained from $F$ by a repetition
of this operation so that each original node has $k$ copies. Then
$t_\ind(F,F(k))\ge 1/n^n$. Let $W$ be the limit graphon of some
subsequence of the $F(k)$ ($k\to\infty$), then $W\in \overline{\PP}$.
Furthermore, clearly $t_\ind(F,W)>0$, and so $F\in\II(W)$.

Conversely, assume that $F=(V,E)\in\II(W)$ for some
$W\in\overline{\PP}$, so that $t_\ind(F,W)>0$. Let $F'$ and $F''$ be
the two graphs obtained from $F$ by doubling a node $v$ ($vv'\notin
E(F')$, but $vv'\in E(F'')$), then
\[
\int\limits_{J^V} t_\ind(F,W;x)\,dx>0
\]
implies that there is a positive measure of choices for the values of
$x_u$ $(u\in V(F)\setminus v)$, for which the set $X$ of the choices
of $x_v$ with $t_\ind(F,W;x)>0$ has positive measure. Clearly either
$W(x,y)<1$ for a positive measure of choices of $(x,y)\in Y$ or this
holds for $W(x,y)>0$. One or the alternative, say the first one,
holds for a positive measure of choices for the values of $x_u$
$(u\in V(F)\setminus v)$. But then $t(F',W)>0$.
\end{proof}

All of the above notions and simple facts extend to bigraphs and
bigraphons trivially.

Let us turn to thin graphons and bigraphons. The significance of thin
bigraphons is supported by the following observation:

\begin{prop}\label{PROP:HERED-THIN}
Let $\PP$ be a hereditary bigraph property that does not contain all
bigraphs. Then every bigraphon in its closure is thin.
\end{prop}

Proposition \ref{PROP:HERED-THIN} and Theorem \ref{THM:FINDIM} imply:

\begin{corollary}\label{COR:HER-PROP}
Let $\PP$ be a hereditary bigraph property that does not contain all
bigraphs. Then for every pure bigraphon $(J_1,J_2,,W)$ in its
closure, $W$ is $0$-$1$ valued almost everywhere, and $J_1$ and $J_2$
are compact and their dimension is bounded by a finite number
depending on $\PP$ only.
\end{corollary}

By this corollary, we can define, for every nontrivial hereditary
property of bigraphs, a finite dimension. It would be interesting to
find further combinatorial properties of this dimension.

The natural analogue of this corollary for graph properties fails to
hold.

\begin{example}\label{EXA:TRI-FREE}
Let $\PP$ be the property of a graph that it is triangle-free. Then
every bipartite graphon is in its closure, but such graphons need not
be $0$-$1$ valued, and their topology need not be finite dimensional
or compact.
\end{example}

However, if we include the (seemingly) simplest of the conclusions of
Corollary \ref{COR:HER-PROP} as a hypothesis, then we can extend it
to all graphs. A graph property $\PP$ is {\it random-free}, if every
$W\in\overline{\PP}$ is $0$-$1$ valued almost everywhere.

\begin{theorem}\label{THM:RF}
Let $\PP$ be a hereditary random-free graph property. Then for every
pure graphon $(J,W)$ in its closure, $J$ is compact and finite
dimensional.
\end{theorem}

Before proving this theorem, we need some preparation.

\begin{lemma}\label{LEM:RF}
For a hereditary graph property $\PP$, the following are equivalent:

\smallskip

{\rm(i)} $\PP$ is random-free;

\smallskip

{\rm(ii)} there is a bigraph $F$ such that $\tbp(F,W)=0$ for all
$W\in\overline{\PP}$;

\smallskip

{\rm(iii)} there is a bipartite graph $F$ with bipartition
$(U_1,U_2)$ such that no graph obtained from $F$ by adding edges
within $U_1$ and $U_2$ has property $\PP$.
\end{lemma}

\begin{proof}
(i)$\Rightarrow$(iii): Assume that (iii) does  not hold, then for
every bigraph $F$ there is a graph $\hat F\in\PP$ and a partition
$V(\hat F)=\{U_1(\hat F),U_2(\hat F)\}$ such that the bigraph between
$U_1(\hat F)$ and $U_2(\hat F)$ is isomorphic to $F$. We want to show
that $\PP$ is not random-free.

\smallskip

Let $(F_1,F_2,\dots)$ be a quasirandom sequence of bigraphs with edge
density $1/2$, with the same number of nodes in each bipartition
class. Consider the graphs $\hat F_n$, and let $F_n'$ and $F_n''$
denote the subgraphs of $\hat F_n$ induced by $U_1(\hat F_n)$ and
$U_2(\hat F_n)$, respectively. By selecting a subsequence we may
assume that the graph sequences $(F'_1,F'_2,\dots)$
$(F''_1,F''_2,\dots)$ are convergent. By Lemma 4.16 in \cite{BCLSV1},
we can order the nodes of $F'_n$ so that $W_{F_n'}$ converges to a
graphon $([0,1],W')$ in the cut norm $\|.\|_\square$, and similarly,
$W_{F_n''}$ converges to a graphon $([0,1],W'')$ in the cut norm. We
order the nodes of $\hat F_n$ so that the nodes in $F_n'$ preceed the
nodes of $F_n''$, and keep the above ordering otherwise. Then
trivially $W_{\hat F_n}$ converges to the graphon
\[
U(x,y)=
  \begin{cases}
    W'(2x,2y) & \text{if $x,y< 1/2$}, \\
    W''(2x-1,2y-1) & \text{if $x,y> 1/2$}, \\
    1/2 & \text{otherwise}.
  \end{cases}
\]
So $U\in\overline{\PP}$ is not $0$-$1$ valued, and so $\PP$ is not
random-free.

\smallskip

(ii)$\Rightarrow$(i): Suppose that $\PP$ is not random-free, and let
$(J,W)\in\overline{\PP}$ be a graphon that is not $0$-$1$ valued
almost everywhere. Then by Theorem \ref{THM:FINDIM}, it is not thin
as a bigraphon, which means that for every bigraph $F=(U_1,U_2,E)$,
$\tbp_\ind(F,W)>0$, so (ii) is not satisfied.

\smallskip

(iii)$\Rightarrow$(ii): Consider a bigraph $F=(U_1,U_2,E)$ as in
(iii), and consider it as a bipartite graph on $V=U_1\cup U_2$ (we
assume that $U_1\cap U_2=\emptyset$). Suppose that it does not
satisfy (ii), then there is a graphon $W\in \overline{\PP}$ such that
$t(F,W;x)>0$ for a positive measure of choices of the $x\in J^V$. For
every such choice, we define a graph $F'$ by connecting those pairs
$\{i,j\}$ of nodes of $F$ for which $W(x_i,x_j)>0$ and either $i,j\in
U_1$ or $i,j\in U_2$. The same supergraph $F'$ will occur for a
positive measure of choices of the $x_i$, and for this $F'$ we have
$t_\ind(F',W)>0$, so using \eqref{EQ:PPII}, we get $F'\in
\II(W)\subseteq\PP$, a contradiction.
\end{proof}

\begin{proof*}{Theorem \ref{THM:RF}}
By Lemma \ref{LEM:RF}, there is a bigraph $F$ such that $\tbp(F,W)=0$
for all $W\in\overline{\PP}$. Thus Theorem \ref{THM:FINDIM} implies
the assertion.
\end{proof*}

\section{Regularity partitions}\label{REGPART}

\subsection{Weak and strong regularity partitions}

The Regularity Lemma of Szemer\'edi \cite{Szem1,Szem2}, and various
weaker and stronger versions of it are basic tools in the study of
large graphs and graphons \cite{LSz3}. Our goal is to show that it is
also closely related to the topology of graphons.

Let $(J,W)$ be a graphon and $\PP$, a partition of $J$ into
measurable sets with positive measure. For $x\in J$, let $S(x)$
denote the partition class containing $x$. Define
\[
f_\PP(x)=\frac1{\pi(S(x))}\int\limits_{S(x)} f(x)\,dx
\]
for a function $f\in L_1(J)$, and
\[
W_\PP(x,y)=\frac1{\pi(S(x))\pi(S(y))}\int\limits_{S(x)\times S(y)}
W(x,y)\,dx.
\]
We say that $\PP$ is a {\it weak regularity partition} with error
$\eps$, if $\|W-W_\PP\|_\square\le\eps$.

We define a {\it Szemer\'edi partition} of a graphon with error
$\eps$ as a partition $\PP=\{S_1\cup\dots\cup S_k\}$ of $J$ into
measurable sets such that
\begin{equation}\label{EQ:SZ-PART}
|\langle W-W_\PP, H\rangle|\le\eps
\end{equation}
for every function $H:~J\times J\to[0,1]$ that is $0$-$1$ valued and
whose support is the union of product sets $R_{ij}=R'_{ij}\times
R_{ij}'\subseteq S_i\times S_j$ ($i,j\in[k]$). To relate this to the
weak partitions, we note that $\|W-W_\PP\|_\square\le\eps$ can be
expressed as \eqref{EQ:SZ-PART} for all functions $h$ of the form
$\one_{S\times T}$. (The formulation above is not a direct
generalization of Szemer\'edi's definition, but it is closest in our
setting; cf. \cite{LSz3}.)

A {\it strong regularity partition} of a graph was introduced by
Alon, Fischer, Krivelevich and M.~Szegedy \cite{AFKS}. Here the error
is specified by an infinite sequence
$\mathbf\EE=(\eps_0,\eps_1,\dots)$ of positive numbers. Again
recasting it in our setting, $\PP$ is a strong regularity partition
with error $\mathbf\EE$ of a graphon $(J,W)$ if there is a graphon
$(J,U)$ such that
\[
\|W-U\|_1\le\eps_0 \qquad\text{and}\qquad
\|U-W_\PP\|_\square\le\eps_{|\PP|}.
\]

Even stronger would be, of course, to require that $\|W-W_\PP\|_1\le
\eps$  (equivalently, \eqref{EQ:SZ-PART} holds for all measurable
functions $H:~J\times J\to[-1,1]$). In this case we call $\PP$ an
{\it ultra-strong regularity partition} with error $\eps$.

The following result is a graphon version of the original
Szemer\'edi's Regularity Lemma \cite{Szem1,Szem2}, its ``weak'' form
due to Frieze and Kannan \cite{FK}, and its strong form due to Alon,
Fischer, Krivelevich and M.~Szegedy \cite{AFKS}. It was proved for
graphons in \cite{LSz3}.

\begin{theorem}\label{THM:FRIEZE-KANNAN}
Let $(J,W)$ be a graphon on an atomfree probability space. Then

\smallskip

{\rm(a)} for every $\eps>0$ $(J,W)$ has a Szemer\'edi partition with
error $\eps$ into no more than $T(\eps)$ classes, where $T(\eps)$
depends only on $\eps$;

\smallskip

{\rm(b)}  for every $\eps>0$ $(J,W)$ has a weak regularity partition
with error $\eps$ into no more than $2^{2/\eps^2}$ classes.

\smallskip

{\rm(c)} for every sequence $\mathbf{\EE}=(\eps_0,\eps_1,\dots)$ of
positive numbers, $(J,W)$ has a strong regularity partition of
$(J,W)$ with error $\mathbf{\EE}$ into no more than $T({\mathbf\EE})$
classes, where $T({\mathbf\EE})$ depends only on ${\mathbf\EE}$.
\end{theorem}

\begin{remark}\label{REM:BALANCED}
(i) We note that every graphon has an ultra-strong partition with
error $\eps$ by standard results in analysis, but the number of
classes cannot be bounded uniformly by any function of $\eps$.

\smallskip

(ii) In the usual formulation, partitions in the Regularity Lemma are
equitable, i.e., the partition classes are as equal as possible. For
graphons on atomless probability spaces, the classes can be required
to have the same measure. In fact, it is easy to see that the
partitions constructed e.g. in Corollary \ref{COR:REG-PART} and
Theorem \ref{THM:SZEMER} below can be repartitioned so that the
classes will be as equal as possible, the error is at most doubled,
and the number of classes is increased by a factor of at most $\lceil
1/\eps\rceil$.
\end{remark}

Several other analytic aspects and versions of the Regularity Lemma
were proved in \cite{LSz3}. One of these results made a connection
between regularity partitions and partitions of $J$ into sets with
small diameter in the $r_{W\circ W}$ metric. Here we prove a
stronger, cleaner version of that result, and then show how to
combine it with our results on thin graphons to get better bounds on
the number of partition classes in weak regularity partitions of this
graphons.

\subsection{Voronoi cells and regularity partitions}

We show that Voronoi cells in the metric spaces $(J,R_W)$ and
$(J,R_{W\circ W})$ are intimately related to different versions of
the Regularity Lemma.

Let $(J,d)$ be a metric space and let $\pi$ be a probability measure
on its Borel sets. We say that a set $S\subseteq J$ is an {\it
average $\eps$-net}, if $\int_J d(x,S)\,d\pi(x) \le\eps$.

Let $S\subseteq J$ be a finite set and $s\in S$. The {\it Voronoi
cell} of $S$ with center $s$ is the set of all points $x\in J$ for
which $d(x,s)\le d(x,y)$ for all $y\in S$. Clearly, the Voronoi cells
of $S$ cover $J$. (We can break ties arbitrarily to get a partition.)

\begin{theorem}\label{THM:VORONOI}
Let $(J,W)$ be a graphon, and let $\eps>0$.

\smallskip

{\rm (a)} Let $S$ be an average $\eps$-net in the metric space
$(S,r_{W\circ W})$. Then the Voronoi cells of $S$ form a weak
regularity partition with error at most $8\sqrt{\eps}$.

\smallskip

{\rm (b)} Let $\PP=\{J_1,\dots,J_k\}$ be a weak regularity partition
with error $\eps$. Then there are points $v_i\in J_i$ such that the
set $S=\{v_1,\dots,v_k\}$ is an average $(4\eps)$-net in the metric
space $(S,r_{W\circ W})$.
\end{theorem}

\begin{proof}
(a) Let $\PP$ be the partition into the Voronoi cells of $S$. Let us
write $R=W-W_\PP$. We want to show that $\|R\|_\square\le
8\sqrt{\eps}$. It suffices to show that for any 0-1 valued function
$f$,
\begin{equation}\label{F-GOAL}
\langle f, Rf \rangle\le 2\sqrt{\eps}.
\end{equation}
Let us write $g=f-f_\PP$, where $f_\PP(x)$ is obtained by replacing
$f(x)$ by the average of $f$ over the class of $\PP$ containing $x$.
Clearly $\langle f_\PP, Rf_\PP \rangle=0$, and so
\begin{equation}\label{EQ:FRF}
\langle f, Rf \rangle = \langle g, R f\rangle + \langle f_\PP, R
f\rangle = \langle f, R g\rangle + \langle f_\PP, R g\rangle \le 2\|R
g\|_1 \le 2\|R g\|_2.
\end{equation}
For each $x\in J$, let $\varphi(x)\in S$ be the center of the Voronoi
cell containing $x$, and define $W'(x,y)=W(x,\phi(y))$ and similarly
$R'(x,y)=R(x,\phi(y))$. Then using that $(W-R)g=W_\PP g=0$,
$W-W'=R-R'$ and $R'g=0$, we get
\begin{align*}
\|Rg\|_2^2&= \langle Rg, Rg \rangle = \langle Wg, (R-R')g\rangle
=\langle Wg, (W-W')g\rangle=\langle g, W(W-W')g\rangle\\
&\le \|W(W-W')\|_1 = \int\limits_{J^2} \Bigl|\int\limits_J
W(x,y)(W(y,z)-W(y,\varphi(z))\,dy\Bigr|\,dx\,dz \\
&= \int\limits_J  r_W(z,\varphi(z)) = \E_\Xb(r_W(\Xb,S)) \le \eps.
\end{align*}
This proves (\ref{F-GOAL}).

\medskip

(b) Suppose that $\PP$ is a weak Szemer\'edi partition with error
$\eps$. Let $R=W-W_\PP$, then we know that $\|R\|_\square \le \eps$.

For every $x\in[0,1]$, define
\[
F(x)=\int\limits_J \Bigl|\int\limits_J  R(x,y)W(y,z)\,dy\Bigr|\,dz.
\]
Then we have
\[
\int\limits_J  F(x)\,dx = \int\limits_{J^3}  s(x,z)
R(x,y)W(y,z)\,dx\,dy\,dz,
\]
where $s(x,z)$ is the sign of $\int R(x,y)W(y,z)$. For every $z\in
J$,
\[
\int\limits_{J^2} s(x,z) R(x,y)W(y,z)\,dx\,dy \le 2\|R\|_\square \le
2\eps,
\]
and hence
\begin{equation}\label{EQ:F-INT}
\int\limits_J F(x)\,dx \le 2\eps.
\end{equation}

Let $x,y\in J$ be two points in the same partition class of $\PP$.
Then $W_\PP(x,s)=W_\PP(y,s)$ for every $s\in J$, and hence
\begin{align*}
r_W(x,y) &= \int\limits_J \Bigl|\int\limits_J
(W(x,s)-W(y,s))W(s,z)\,ds\Bigr|\,dz\\
&=\int\limits_J \Bigl|\int\limits_J
(R(x,s)-R(y,s))W(s,z)\,ds\Bigr|\,dz\\
&\le \int\limits_J \Bigl|\int\limits_J  R(x,s)W(s,z)\,ds\Bigr|\,dz+
\int\limits_J \Bigl|\int\limits_J
R(y,s)W(s,z)\,ds\Bigr|\,dz\\
&=F(x)+F(y).
\end{align*}
For every set $T\in\PP$, let $v_T\in T$ be a point ``below average''
in the sense that
\[
F(v_T)\le \frac1{\pi(T)}\int\limits_T F(x)\,dx,
\]
and let $S=\{v_T:~T\in\PP\}$. Then using \eqref{EQ:F-INT},
\begin{align*}
\E_\Xb d(\Xb,S) &\le \sum_{T\in\PP} \int\limits_T d(x,v_T)\,dx \le
\sum_{T\in\PP} \int\limits_T (F(x)+F(v_T))\,dx\\
&\le \int\limits_J  F(x)\,dx +\sum_{T\in\PP} \lambda(T) F(v_T) \le
2\int\limits_J F(x)\,dx\le 4\eps.
\end{align*}
This proves the Theorem.
\end{proof}

Theorems \ref{THM:VORONOI} and \ref{THM:FINDIM} imply the following
Corollary (we prove a stronger result in the next section).

\begin{corollary}\label{COR:REG-PART}
For every bigraph $F=(V,E)$ there is a constant $c_F>0$ such that if
$G$ is a graph not containing $F$ as an induced sub-bigraph, then for
every $\eps>0$, $G$ has a weak regularity partition with error $\eps$
with at most $c_F\eps^{-10|V|}$ classes.
\end{corollary}

\begin{remark}\label{REM:NONBIP}
The conclusion does not remain true if the subgraph we exclude is
nonbipartite. Any bipartite graph will then satisfy the condition,
and some bipartite graphs are known to need an exponential (in
$1/\eps$) number of classes in their weak regularity partitions.
\end{remark}

\subsection{Edit distance}

We conclude with deriving bounds on the size of the Szemer\'edi
partitions and approximations in $L_1$, using the packing dimension
of $(J, r_W)$. In the graph theoretic case, this corresponds to
approximation in edit distance.

\begin{lemma}\label{unbalanced}
Let $W$ be a graphon such that $(J,r_W)$ can be covered by $m$ balls
of radius $\eps$. Then there is a stepfunction $U$ with $m(1/\eps)^m$
steps such that $\|W-U\|_1\leq 2\eps$.
\end{lemma}

\begin{remark}\label{01}
If $W$ is $0$-$1$ valued, then the bound on the number of classes can
be improved to $m2^m$.
\end{remark}

\begin{proof}
Let $\PP=\{J_1,J_2,\dots,J_m\}$ be a partition of $J$ into measurable
sets such that for every $i$ there is $x_i\in J$ with
$\|W(x_i,.)-W(x,.)\|_1\leq\eps$ for every $x\in J_i$. Let
$W'(x,y)=W(x_i,y)$ for $x\in J_i$, then trivially
$\|W-W'\|_1\le\eps$. Let $\QQ_i$ be a partition of $J$ into $1/\eps$
measurable classes so that $W(x_i,.)$ varies at most $\eps$ on each
class of $\QQ_i$. For $x\in J_i$ and $y\in S\in\QQ_i$, define
\[
U(x,y)=\frac1{\pi(S)} \int_S W'(x,z)\,dz.
\]
Then clearly $|U(x,y)-W'(x,y)|\le\eps$ for all $x,y\in J$, and hence
$\|U-W\|_1 \le \|U-W'\|_1 + \|W-W'\|_1\le 2\eps$. It is obvious that
$U$ is a stepfunction in the partition generated by $\PP$ and
$\QQ_1,\dots,\QQ_m$, which has at most $m(1/\eps)^m$ classes.
\end{proof}

We obtain from this lemma:

\begin{theorem}\label{THM:SZEMER}
Let $W$ be a graphon such that $(J,r_W)$ has packing dimension $d$,
then for every $0<\eps<1$ it has an ultra-strong partition with error
$\eps$ and with at most $\eps^{-O(\eps^{-d})}$ classes.
\end{theorem}

\begin{proof}
Consider a maximal packing in $(J,r_W)$ of balls with radius
$\eps/8$; this consists of $m=O(\eps^{-d})$ balls. The balls with the
same centers and with radius $\eps/4$ cover $J$, so Lemma
\ref{unbalanced} there is a stepfunction $U$ with $m(4/\eps)^m\le
\eps^{-c\eps^{-d}}$ steps such that $\|W-U\|_1\leq \eps/2$. For the
partition $\PP$ into the steps of $U$, we have
\[
\|W-W_\PP\|_1\le 2\|W-U\|_1\le\eps
\]
(the first inequality follows by easy computation).
\end{proof}

For thin graphons, we get a stronger bound.

\begin{theorem}\label{THM:THIN-ULTRA}
Let $W$ be a thin graphon in which a bigraph $F=(V,E)$ is excluded as
an induced sub-bigraph. Then for every $0<\eps<1$, it has an
ultra-strong partition with error $\eps$ and with $O(\eps^{-10
|V|^2})$ classes.
\end{theorem}

\begin{proof}
Theorem \ref{THM:FINDIM} implies that $W$ is $0$-$1$ valued and it
has finite packing dimension at most $10|V|$. Similarly to the proof
of lemma \ref{unbalanced}, let $\PP=\{J_1,J_2,\dots,J_m\}$ be a
partition of $J$ with $m=O(\eps^{-|V|})$ into measurable sets such
that for every $i$ there is an $x_i\in J$ with
$\|W(x_i,.)-W(x,.)\|_1\leq\eps$ for every $x\in J_i$. Let
$W'(x,y)=W(x_i,y)$ for $x\in J_i$, then $\|W'-W\|_1\leq\eps$. Let
$S_i$ be the support of the function $W(x_i,.)$, and let $A$ be the
set of atoms of the Boolean algebra generated by
$\{S_1,S_2,\dots,S_m\}$ with positive measure. For every atom $a\in
A$, let $F_a\subseteq [m]$ denote the index set $\{i|a\subseteq
S_i\}$ and let $\mathcal{F}$ denote the set system $\{F_a|a\in A\}$.
Since $F$ is not an induced sub-bigraph, $\mathcal{F}$ has
VC-dimension less than $|V|$, and so by lemma \ref{LEM:SaSh} we
obtain that $|A|\leq O(m^{|V|-1})$. The joint refinement
$\mathcal{P}_2$ of $A$ and $\PP$ is of size at most
$O(\eps^{-10|F|^2})$. This completes the proof since $W'$ is a
stepfunction with steps in $\mathcal{P}_2$.
\end{proof}

It is easy to see that in the definition of ultra-strong regularity
partitions of $0$-$1$ valued graphons, we can replace $W_\PP$ by a
$0$-$1$ valued stepfunction with the same steps, at the cost of
doubling the error. Together with Remark \ref{REM:BALANCED}, we can
apply this to a (large) finite graph $G$. To state the result, we
need a definition. Let $H$ be a simple graph, and let us replace each
node $v$ of $H$ by a set $S_v$ of ``twin'' nodes, where two nodes
$x\in S_u$ and $y\in S_v$ are connected if and only if $uv\in E(H)$.
For each $u\in V(H)$, either connect all pairs of nodes in $S_u$, or
none of them. We call every graph obtained this way a {\it blow-up}
of $H$.

\begin{corollary}\label{COR:L1}
For every bigraph $F$ there is a constant $c_F>0$ such that if $G$ is
a graph not containing $F$ as an induced sub-bigraph, then for every
$\eps>0$, we can change $\eps|G|^2$ edges of $G$ so that the
resulting graph is a blow-up of a graph with at most
$c_F\eps^{-10|F|^2}$ nodes.
\end{corollary}

Let us say that a graphon $W$ has {\it polynomial $L_1$-complexity}
if there is a $d>0$ such that for every $\eps>0$ there is a
stepfunction $W'$ with $O(\eps^{-d})$ steps satisfying
$\|W-W'\|_1\leq\eps$. We can define  {\it polynomial
$\square$-complexity} analogously. As we have pointed out, polynomial
$\square$-complexity corresponds to the finite dimensionality of the
metric space of $W\circ W$. Theorem \ref{THM:THIN-ULTRA} implies that
every thin graphon has polynomial $L_1$-complexity.

If $W$ has polynomial complexity, then the structure of $W$ can be
described by a polynomial number (in $1/\eps$) of real parameters
with an error $\eps$ in the appropriate norm. The set of graphons
with polynomial complexity is closed under many natural operations
such as operator product, tensor product, etc.

It could be interesting to study other aspects of this complexity
notion. We offer a conjecture relating our complexity notion to
extremal combinatorics. It is supported by examples in \cite{LSz6}.

\begin{conj}
Let $F_1,F_2,\dots,F_n$ be a set of finite graphs,
$t_1,t_2,\dots,t_m$ be real numbers in $[0,1]$ and $\mathcal{S}$ be
the set of graphons $W$ with $t(F_i,W)=t_i$ for $1\leq i\leq n$. Then
$\mathcal{S}$ is either empty or it contains a graphon of polynomial
$L_1$-complexity.
\end{conj}

\end{document}